\documentclass[11pt, a4paper]{amsart}
\usepackage[margin=1.5cm]{geometry}
\usepackage{amsthm}
\usepackage[colorlinks=true,citecolor=cyan,backref=page]{hyperref}
\usepackage{amsmath}
\usepackage{amsfonts}
\usepackage{amssymb}
\usepackage{stmaryrd}  
\usepackage{mathrsfs} 
\usepackage{esint} 
\usepackage{array} 
\usepackage{latexsym}
\usepackage{lmodern} 
\usepackage{graphicx} 
\usepackage{graphics} 
\usepackage{epsfig} 
\usepackage{color} 
\usepackage[dvipsnames]{xcolor}
\usepackage[all]{xy}
\usepackage[new]{old-arrows}
\usepackage{textcomp}
\usepackage{stmaryrd}
\DeclareGraphicsExtensions{.eps,.pdf,.jpeg,.png} 
\usepackage{fancyhdr}
\usepackage{multirow} 
\usepackage{comment}
\usepackage{pgf,tikz}
\usepackage{paralist}
\usepackage{faktor}
\usepackage{dirtytalk}
\usepackage{url}
\usepackage{blkarray,bigstrut}
\usepackage{nicematrix}
\usepackage{enumerate}
\usepackage{pgfplots}
\usepackage[capitalize]{cleveref}

\Crefname{equation}{Formula}{Formulas}

\makeatletter
\renewcommand*\env@matrix[1][*\c@MaxMatrixCols c]{%
  \hskip -\arraycolsep
  \let\@ifnextchar\new@ifnextchar
  \array{#1}}
\makeatother

\newtheorem{proposition}{Proposition}[section]

\newtheorem{conjecture}{Conjecture}[section]
\newtheorem{lemma}{Lemma}[section]

\theoremstyle{definition}

\newtheorem{remark}{Remark}[section]
\newtheorem{example}{Example}[section]

\newtheorem{question}{Question}[section]

\makeatletter
\DeclareRobustCommand*{\mfaktor}[3][]
{
   { \mathpalette{\mfaktor@impl@}{{#1}{#2}{#3}} }
}
\newcommand*{\mfaktor@impl@}[2]{\mfaktor@impl#1#2}
\newcommand*{\mfaktor@impl}[4]{
   \settoheight{\faktor@zaehlerhoehe}{\ensuremath{#1#2{#3}}}%
   \settoheight{\faktor@nennerhoehe}{\ensuremath{#1#2{#4}}}%
      \raisebox{-0.5\faktor@zaehlerhoehe}{\ensuremath{#1#2{#3}}}%
      \mkern-4mu\diagdown\mkern-5mu%
      \raisebox{0.5\faktor@nennerhoehe}{\ensuremath{#1#2{#4}}}%
}
\makeatother

\title{A counterexample to a Brenti-Carnevale conjecture}

\author{
Nathan Chapelier-Laget, Jean Fromentin
}

\begin{document}

\maketitle

\begin{abstract}
Recently, F. Brenti put a preprint on the arXiv with several interesting open problems on Coxeter groups and unimodality \cite{brenti2024conj}. In this note, we refute one of these conjectures with a counterexample and provide supporting data related to it. This work serves as an initial step toward further exploration of the topic.
\end{abstract}

\section{General background}

\subsection{Finite Coxeter groups}
Let $(W,S)$ be a finite Coxeter system with length function $\ell$. Let $\Phi = \Phi^+ \sqcup \Phi^-$ be the root system of $W$ with simple system $\Delta$. Let $\alpha \in \Phi$ such that  $\alpha = a_1\alpha_1 + \cdots + a_n\alpha_n$ with $a_i \in \mathbb{Z}$. The height of $\alpha$ (with respect to $\Delta$) is defined by $\mathrm{ht}(\alpha) = a_1 + \cdots+ a_n$. Clearly $\mathrm{ht}(\alpha + \beta) = \mathrm{ht}(\alpha) + \mathrm{ht}(\beta)$ and $\mathrm{ht}(-\alpha)=-\mathrm{ht}(\alpha)$ for all $\alpha, \beta, \alpha + \beta \in \Phi$. We let $w_0$ be the longest element of $W$. This is the only element in $W$ satisfying $w_0(\Phi^+) = \Phi^-$, or equivalently $w_0(\Delta) = -\Delta$. Moreover $w_0^{-1} = w_0$, and for any $\alpha \in \Phi$ we have $\mathrm{ht}(w_0(\alpha)) = -\mathrm{ht}(\alpha)$.
Let $T$ be the set of reflections of $W$, that is $T = \{wsw^{-1}~|~s \in S, w \in W\}$. There is a natural bijection between $T$ and $\Phi^+$, and for 
$\alpha \in \Phi^+$, we denote by $s_{\alpha} \in T$ the corresponding reflection.

\medskip

The left-inversion set and left-descent set of an element $w \in W$ are defined respectively by 
\begin{align*}
    N_L(w)  &:= \{t \in T ~|~\ell(tw) < \ell(w)\} \\
     D_L(w) & : = \{s \in S ~|~\ell(sw) < \ell(w)\}.
\end{align*}
Similarly, the right-inversion set and right-descent set are defined by 
\begin{align*}\label{right Inv set}
    N_R(w)  &:= \{t \in T ~|~\ell(wt) < \ell(w)\}\\
    D_R(w) & : = \{s \in S ~|~\ell(ws) < \ell(w)\}.
\end{align*}

\medskip

Let $I,J \subset S$ and let $(W_I,S_I)$ and $(W_J,S_J)$ be the corresponding parabolic subgroups of $W$. We denote by  
$$
{}^JW = \{w \in W~|~\ell(sw) > \ell(w)~\forall s \in J\}
$$
and
 $$
 W^I = \{w \in W~|~\ell(ws) > \ell(w)~\forall s \in I\}.
 $$
 The set ${}^JW $ is called the $J$-left transverse set of $W$ and $W^I$ is called the $I$-right transverse set of $W$. Another description of the transverse sets can be found in \cite[Section 2.5]{SRLE}.

Each element $w \in W$ has unique decomposition $w = w_J{}^Jw$ with $w_J \in W_J$, ${}^Jw \in {}^JW $ called the left $J$-decomposition of $w$, and has also a unique decomposition  $w = w^Iw_I{}$ with $w_I \in W_I$, $w^I \in W^I $ called the right $I$-decomposition of $w$.  These decompositions satisfy 
$$
\ell(w) = \ell(w_J) + \ell({}^Jw) = \ell(w^I) + \ell(w_I).
$$
The element ${}^Jw$ is the unique element of minimal length in $W_Jw$ and $w^I$ is the unique element of minimal length in $wW_I$.  

\bigskip

\subsection{The odd-length in type $A_n$}\label{section type A}
Let $W(A_n)$ be the finite Coxeter system of type $A_n$, with set of simple generators $S = \{s_i~|~i =1,\dots, n\}$ where $s_i$ is the adjacent transposition $(i,i+1)$. Let $(e_i)_i$ be the canonical basis of $\mathbb{R}^{n+1}$. A way to describe the $\frac{n(n+1)}{2}$ roots of $W(A_n)$ is by
$
\Phi=\{\pm (e_i - e_j) ~| ~ 1\leq i < j \leq n+1 \},
$
with simple system 
$\Delta = \{ e_i - e_{i+1} ~| ~ 1\leq i <  n +1\},$
and  positive roots 
$$
\Phi^+=\{e_i - e_j ~| ~ 1\leq i < j \leq n+1 \}.
$$

For short, we right $e_{ij}:=e_i - e_j$. It is easy to see, for any $e_{ij} \in \Phi^+$, that
\begin{align}
    \mathrm{ht}(e_{ij}) = j-i.
\end{align}

Let $\mathcal{I}(w)$ be the \say{usual} inversion set of $w \in W(A_n)$, namely
\begin{align}
    \mathcal{I}(w):= \{(i,j) \in [n+1]^2~|~i<j~\text{~and~}~w(i)>w(j)\}.
\end{align}

The group $W(A_n)$ acts on $\Phi$ by $w(e_{i,j}) = e_{w(i),w(j)}$. Therefore, we have a natural identification between $N_R(w)$ and $\mathcal{I}(w)$, which is defined by $e_{ij} \mapsto (i,j)$.

\bigskip

In \cite[Definition 5.1]{klopsch2009igusa}, B. Klopsch and C. Voll introduced the statistic $L$ as follows
\begin{equation}\label{odd-length type A}
\begin{array}{ccccc}
      L & : & W(A_n) & \longrightarrow & \mathbb{N}  \\
                    &   & w & \longmapsto     & \sum\limits_{J \subset S}(-1)^{|J|}2^{|S|-|J|-1}\ell({}^Jw).
\end{array}
\end{equation}

Note that, a priori, it is not immediately clear that 
$L$ takes only nonnegative values. However, this follows as a consequence of the next lemma.

\begin{lemma}{\cite[Lemma 5.2]{klopsch2009igusa}}
    Let $w \in W(A_n)$. We have
    \begin{equation}
        L(w) = |\{(i,j) \in \mathcal{I}(w)~|~j-i = 1 ~\mathrm{mod} ~2\}|.
    \end{equation}
\end{lemma}

\bigskip

\subsection{The Brenti-Carnevale conjecture} The map $L$ was first introduced as a tool to study some symmetries enjoyed by polynomials enumerating nondegenerate flags in finite vector spaces. More recently, the map $L$ has been further investigated \cite{brenti2024conj, brenti2017proof, brenti2021odd}, leading, among others, to some interesting properties and conjectures.

\medskip

\begin{proposition}{\cite[Proposition 2.2]{brenti2021odd}}\label{prop propriété type A}
The map $L$ satisfies the following properties
    \begin{itemize}
        \item[(1)] $L(e) = 0$,
          \item[(2)] $L(s) = 1$ for any $s \in S$,
        \item[(3)] $L(w_0w) = L(ww_0) = L(w_0)-L(w)$ for any $w \in W(A_n)$,
        \item[(4)] $w_0$ is the unique element on which $L$ reaches its maximum $L(w_0) = \lfloor \frac{n+1}{2} \rfloor \lceil \frac{n+1}{2} \rceil$.
    \end{itemize}
\end{proposition}

\medskip
A polynomial $P$ is called \textit{unimodal} if the sequence of its coefficients $(a_n)_n$ is unimodal, that is, if there exists $N \in \mathbb{N}$ such that for any $0\leq n \leq N-1$, $a_n \leq a_{n+1}$ and for any $N\leq n$, $a_n \geq a_{n+1}$.
\medskip

\begin{conjecture}{\cite[Conjecture 6.2]{brenti2021odd}}\label{conj 1}
For any $n \geq 4$, the following polynomial is unimodal 
$$
B_n(q) := \sum\limits_{w \in W(A_n)}q^{L(w)}.
$$
\end{conjecture}

\medskip

\begin{example}
\begin{itemize}
    \item[] 
    \item ~ $B_1(q) = 1 + q$.
    \item ~ $B_2(q) = 1 + 4q+ q^2$.
    \item ~ $B_3(q) = 1 + 8q + 6q^2 + 8q^3 + q^4$.
    \item ~ $B_4(q) = 1 + 12q + 23q^2 + 48q^3 + 23q^4 + 12q^5 + 1q^6$.
    \item ~ $B_5(q) = 1 + 16q + 59q^2 + 137q^3 + 147q^4 + 147q^5 + 137q^6 + 59q^7 + 16q^8 +q^{9}$.
    \item ~ $B_6(q) = 1 + 20q + 113q^2 + 300q^3 + 631q^4 + 832q^5 + 1246q^6 + 832q^7 + 631q^8 +300q^{9}+113q^{10}+20q^{11}+q^{12}$.
\end{itemize}
\end{example}

\medskip

F. Brenti and A. Carnevale further extended the concept of inversion sets, by defining for any $k,n \in \mathbb{N}$ and any residue $h ~\text{mod}~ k$, the following set (see \cite[Formula (6.1)]{brenti2021odd})
\begin{align*}
    \mathcal{I}_{k,h}(w):= \{(i,j) \in \mathcal{I}(w)~|~j-i = h~\text{mod}~k\}.
\end{align*}
along with
\begin{align*}
    \mathfrak{inv}_{k,h}(w) := |\mathcal{I}_{k,h}(w)|. 
\end{align*}

\medskip

\begin{conjecture}{\cite[Conjecture 6.3]{brenti2021odd}}\label{conj2}
For any $n \in \mathbb{N}$ and any $k \geq 3$, the following polynomial is unimodal 
$$
B_{n,k,1}(q) := \sum\limits_{w \in W(A_n)}q^{\mathfrak{inv}_{k,1}(w)}.
$$
\end{conjecture}

\bigskip

It is straightforward from the definition that $\mathfrak{inv}_{2,1}(w) = L(w)$. Moreover, it is also easy to see for $k \geq n-2$ that $B_{n,k,1}(q)$ is equal to the Eulerian polynomial, that is
$$
B_{n,k,1}(q) = \sum\limits_{w \in W(A_n)}q^{|D_R(w)|}.
$$

\bigskip

\begin{question}
    We may wonder if there is an analogue for $\mathfrak{inv}_{k,h}$ in terms of a map resembling to (\ref{odd-length type A}) ?
\end{question}

\bigskip

\newpage

\section{Conjecture \ref{conj 1}}

In \cite{brenti2021odd}, the authors have verified the conjecture up to $n=10$, however, it fails at the very next step. We give below some explicit data.

\bigskip

\begin{figure}[h!]
\[
\tiny
\begin{array}{r|r|r|r|r|r|r|r|r|r|r|r|r|r|r}
d & N_{1} & N_{2} & N_{3} & N_{4} & N_{5} & N_{6} & N_{7} & N_{8} & N_{9} & N_{10} & N_{11} & N_{12} & N_{13} & N_{14} \\
\hline
0 & \textcolor{blue}{1} & 1 & 1 & 1 & 1 & 1 & 1 & 1 & 1 & 1 & 1 & 1 & 1 & 1\\
1 & \textcolor{blue}{1} & \textcolor{blue}{4} & 8 & 12 & 16 & 20 & 24 & 28 & 32 & 36 & 40 & 44 & 48 & 52\\
2 & - & 1 & {\color{red}6} & 23 & 59 & 113 & 183 & 269 & 371 & 489 & 623 & 773 & 939& 1121 \\
3 & - & - & 8 & \textcolor{blue}{48} & 137 & 300 & 620 & 1184 & 2056 & 3296 & 4968 & 7136 & 9864 & 13216\\
4 & - & - & 1 & 23 & \textcolor{blue}{147} & 631 & 1878 & 4201 & 8155 & 14823 & 25579 & 41995 & 65887 & 99335\\
5 & - & - & - & 12 & \textcolor{blue}{147} & 832 & 2956 & 8524 & 22273 & 50940 & 102584 & 189532 & 330648 & 551604\\
6 & - & - & - & 1 & 137 & \textcolor{blue}{1246} & 5481 & 18548 & 54124 & 134230 & 302591 & 644400 & 1287220 & 2411214 \\
7 & - & - & - & - & 59 & 832 & 5616 & 28244 & 105147 & 310324 & 817040 & 1950708 & 4241497 & 8648368\\
8 & - & - & - & - & 16 & 631 & \textcolor{blue}{6802} & 42070 & 173888 & 585142 & 1779709 & 4797444 & 11733016 & 26837670 \\
9 & - & - & - & - & 1 & 300 & 5616 & 48420 & 254631 & 1060652 & 3727112 & 11009044 & 29364967 &73107128\\
10 & - & - & - & - & - & 113 & 5481 & \textcolor{blue}{59902} & 350506 & 1626663 & 6383306 & 21371677 & 65173557 & 180535647\\
11 & - & - & - & - & - & 20 & 2956 & 48420 & 402868 & 2440224 & 10999144 & 40427016 & 133646843 & 397580820\\
12 & - & - & - & - & - & 1 & 1878 & 42070 & \textcolor{blue}{440348} & 3116617 & 15876597 & 67176502 & 252077108 & 825351072\\
13 & - & - & - & - & - & - & 620 & 28244 & \textcolor{blue}{440348} & 3940532 & 23094400 & 109067416 & 444000080 & 1563580256\\
14 & - & - & - & - & - & - & 183 & 18548 & 402868 & 4285235 & 28960385 & 158426584 & 724445818 & 2816736057\\
15 & - & - & - & - & - & - & 24 & 8524 & 350506 & \textcolor{blue}{4778392} & 37367608 & 227994816 & 1126089020 & 1563580256\\
16 & - & - & - & - & - & - & 1 & 4201 & 254631 & 4285235 & 40981666 & 296523080 & 1648976512 & 2816736057 \\
17 & - & - & - & - & - & - & - & 1184 & 173888 & 3940532 & 46229504 & 380494144 & 2300776313 & 4728979076 \\
18 & - & - & - & - & - & - & - & 269 & 105147 & 3116617 & {\color{red}45695886} & 444961297 & 3049166295 & 7669867339\\
19 & - & - & - & - & - & - & - & 28 & 54124 & 2440224 & 46229504 & 515222884 & 3875492392 & 11628091516\\
20 & - & - & - & - & - & - & - & 1 & 22273 & 1626663 & 40981666 & 544667447 & 4704025183 & 17190758276 \\
21 & - & - & - & - & - & - & - & - & 8155 & 1060652 & 37367608 & \textcolor{blue}{577072920} & 5523283878 &23892874904\\
22 & - & - & - & - & - & - & - & - & 2056 & 585142 & 28960385 & 544667447 & 6156913063 & 32526492523\\
23 & - & - & - & - & - & - & - & - & 371 & 310324 & 23094400 & 515222884 & 6647055284 & 41831536516\\
24 & - & - & - & - & - & - & - & - & 32 & 134230 & 15876597 & 444961297 & \textcolor{blue}{6890990167}& 52890180033 \\
25 & - & - & - & - & - & - & - & - & 1 & 50940 & 10999144 & 380494144 & \textcolor{blue}{6890990167} &63197491888\\
26 & - & - & - & - & - & - & - & - & - & 14823 & 6383306 & 296523080 & 6647055284 & 74391276942\\
27 & - & - & - & - & - & - & - & - & - & 3296 & 3727112 & 227994816 & 6156913063 & 82890039732\\
28 & - & - & - & - & - & - & - & - & - & 489 & 1779709 & 158426584 & 5523283878 &91181445876\\
29 & - & - & - & - & - & - & - & - & - & 36 & 817040 & 109067416 & 4704025183 & 94927984924\\
30 & - & - & - & - & - & - & - & - & - & 1 & 302591 & 67176502 & 3875492392 & \textcolor{blue}{97989421788}\\
31 & - & - & - & - & - & - & - & - & - & - & 102584 & 40427016 & 3049166295 & 94927984924\\
32 & - & - & - & - & - & - & - & - & - & - & 25579 & 21371677 & 2300776313 & 91181445876 \\
33 & - & - & - & - & - & - & - & - & - & - & 4968 & 11009044 & 1648976512 & 82890039732 \\
34 & - & - & - & - & - & - & - & - & - & - & 623 & 4797444 & 1126089020 & 74391276942 \\
35 & - & - & - & - & - & - & - & - & - & - & 40 & 1950708 & 724445818 & 63197491888\\
36 & - & - & - & - & - & - & - & - & - & - & 1 & 644400 & 444000080 & 52890180033\\
37 & - & - & - & - & - & - & - & - & - & - & - & 189532 & 252077108 & 41831536516\\
38 & - & - & - & - & - & - & - & - & - & - & - & 41995 &  133646843 & 32526492523\\
39 & - & - & - & - & - & - & - & - & - & - & - & 7136 & 65173557  & 23892874904\\
40 & - & - & - & - & - & - & - & - & - & - & - & 773 & 29364967 & 17190758276\\
41 & - & - & - & - & - & - & - & - & - & - & - & 44 & 11733016 &11628091516\\
42 & - & - & - & - & - & - & - & - & - & - & - & 1 & 4241497 &7669867339\\
43 & - & - & - & - & - & - & - & - & - & - & - & - & 1287220 & 4728979076\\
44 & - & - & - & - & - & - & - & - & - & - & - & - & 330648 &2816736057\\
45 & - & - & - & - & - & - & - & - & - & - & - & - & 65887 & 1563580256\\
46 & - & - & - & - & - & - & - & - & - & - & - & - & 9864 & 825351072\\
47 & - & - & - & - & - & - & - & - & - & - & - & - & 939 &397580820\\
48 & - & - & - & - & - & - & - & - & - & - & - & - & 48 & 180535647\\
49 & - & - & - & - & - & - & - & - & - & - & - & - & 1 & 73107128\\
50 & - & - & - & - & - & - & - & - & - & - & - & - & - & 26837670\\
51 & - & - & - & - & - & - & - & - & - & - & - & - & - & 8648368\\
52 & - & - & - & - & - & - & - & - & - & - & - & - & - & 2411214\\
53 & - & - & - & - & - & - & - & - & - & - & - & - & - &551604\\
54 & - & - & - & - & - & - & - & - & - & - & - & - & - & 99335\\
55 & - & - & - & - & - & - & - & - & - & - & - & - & - &13216\\
56 & - & - & - & - & - & - & - & - & - & - & - & - & - &1121\\
57 & - & - & - & - & - & - & - & - & - & - & - & - & - &52\\
58 & - & - & - & - & - & - & - & - & - & - & - & - & - &1\\
\end{array}
\]
\caption{For fixed $n$, we denote by $N_n(d)$ the number of elements $w \in W(A_n)$ with $L(w) = d$. The blue indicates the maximal coefficients reached.}
\label{table}
\end{figure}

\textcolor{blue}{}

\begin{remark}
   It is unlikely that \Cref{conj 1} holds in general for $n \geq 12$. Computational experiments, using the software \textbf{C}++ along with \cite{calculco}, have verified this conjecture up to $n = 17$, revealing that the only cases where it fails are $n = 3$ and $n = 11$. 
   We might imagine that the next case where it fails is $n = 19$. Unfortunately, it is out of reach by computational experiments. We actually believe that there are infinitely many integers $n$ such that $B_n(q)$ is not unimodal. It would be very interesting to  characterize (if true) those numbers.
\end{remark}

\newpage

Below, we display the discrete curves corresponding to Figure \ref{table}.

\bigskip

\begin{figure}[h!]
    \centering
    \begin{minipage}{0.45\textwidth} 
        \centering
        \begin{tikzpicture}
        \begin{axis}
        \addplot table [x=l, y=N]{csv/res_2_5.csv};
        \end{axis}
        \end{tikzpicture}
        \caption{Case $n = 5$.}
    \end{minipage}
    \hfill
    \begin{minipage}{0.45\textwidth}
        \centering
        \begin{tikzpicture}
        \begin{axis}
        \addplot table [x=l, y=N]{csv/res_2_6.csv};
        \end{axis}
        \end{tikzpicture}
        \caption{Case $n = 6$.}
    \end{minipage}
\end{figure}


\begin{figure}[h!]
    \centering
    \begin{minipage}{0.45\textwidth}
        \centering
        \begin{tikzpicture}
        \begin{axis}
        \addplot table [x=l, y=N]{csv/res_2_7.csv};
        \end{axis}
        \end{tikzpicture}
        \caption{Case $n = 7$.}
    \end{minipage}
    \hfill
    \begin{minipage}{0.45\textwidth}
        \centering
        \begin{tikzpicture}
        \begin{axis}
        \addplot table [x=l, y=N]{csv/res_2_8.csv};
        \end{axis}
        \end{tikzpicture}
        \caption{Case $n = 8$.}
    \end{minipage}
\end{figure}


\begin{figure}[h!]
    \centering
    \begin{minipage}{0.45\textwidth}
        \centering
        \begin{tikzpicture}
        \begin{axis}
        \addplot table [x=l, y=N]{csv/res_2_9.csv};
        \end{axis}
        \end{tikzpicture}
        \caption{Case $n = 9$.}
    \end{minipage}
    \hfill
    \begin{minipage}{0.45\textwidth}
        \centering
        \begin{tikzpicture}
        \begin{axis}
        \addplot table [x=l, y=N]{csv/res_2_10.csv};
        \end{axis}
        \end{tikzpicture}
        \caption{Case $n = 10$.}
    \end{minipage}
\end{figure}

\newpage


\begin{figure}[h!]
    \centering
    \begin{minipage}{0.45\textwidth}
        \centering
        \begin{tikzpicture}
        \begin{axis}
        \addplot table [x=l, y=N]{csv/res_2_11.csv};
        \end{axis}
        \end{tikzpicture}
        \caption{Case $n = 11$, refuting \Cref{conj 1}.}
    \end{minipage}
    \hfill
    \begin{minipage}{0.45\textwidth}
        \centering
        \begin{tikzpicture}
        \begin{axis}
        \addplot table [x=l, y=N]{csv/res_2_12.csv};
        \end{axis}
        \end{tikzpicture}
        \caption{Case $n = 12$.}
    \end{minipage}
\end{figure}


\begin{figure}[h!]
    \centering
    \begin{minipage}{0.45\textwidth}
        \centering
        \begin{tikzpicture}
        \begin{axis}
        \addplot table [x=l, y=N]{csv/res_2_13.csv};
        \end{axis}
        \end{tikzpicture}
        \caption{Case $n = 13$.}
    \end{minipage}
    \hfill
    \begin{minipage}{0.45\textwidth}
        \centering
        \begin{tikzpicture}
        \begin{axis}
        \addplot table [x=l, y=N]{csv/res_2_14.csv};
        \end{axis}
        \end{tikzpicture}
        \caption{Case $n = 14$.}
    \end{minipage}
\end{figure}

\begin{figure}[h!]
    \centering
    \begin{minipage}{0.45\textwidth}
        \centering
        \begin{tikzpicture}
        \begin{axis}
        \addplot table [x=l, y=N]{csv/res_2_15.csv};
        \end{axis}
        \end{tikzpicture}
        \caption{Case $n = 15$.}
    \end{minipage}
    \hfill
    \begin{minipage}{0.45\textwidth}
        \centering
        \begin{tikzpicture}
        \begin{axis}
        \addplot table [x=l, y=N]{csv/res_2_16.csv};
        \end{axis}
        \end{tikzpicture}
        \caption{Case $n = 16$.}
    \end{minipage}
\end{figure}

\begin{figure}[h!]
    \centering
    \begin{tikzpicture}
    \begin{axis}
    \addplot table [x=l, y=N]{csv/res_2_17.csv};
    \end{axis}
    \end{tikzpicture}
    \caption{Case $n = 17$.}
\end{figure}
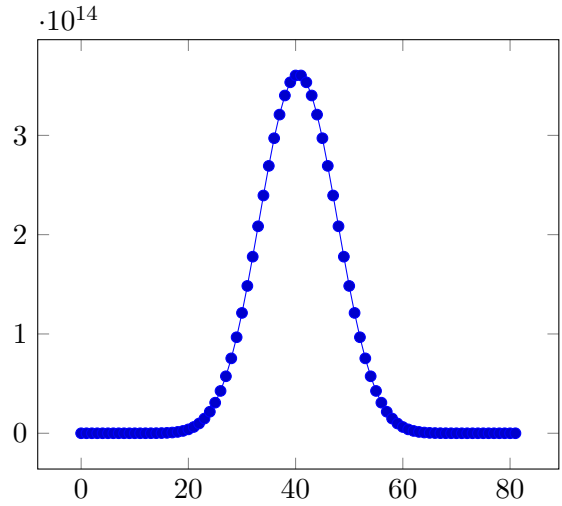

\bigskip

\newpage

\section{ \Cref{conj2}}
 \Cref{conj2} seems to hold, as it is illustrated in the following table, where for fixed $n$ and $k \geq 3$, we denote by $N_n(d)$ the coefficient of $q^d$ in $B_{n,k,1}(q)$.
\begin{figure}[h!]
\[
\begin{array}{r|r|r|r|r|r|r|r|r|r|r|r|r|}
d & N_{1} & N_{2} & N_{3} & N_{4} & N_{5} & N_{6} & N_{7} & N_{8} & N_{9} & N_{10} & N_{11} & N_{12} \\
\hline
0 & \textcolor{blue}{1} & 1 & 1 & 1 & 1 & 1 & 1 & 1 & 1 & 1 & 1 & 1 \\
1 & \textcolor{blue}{1} & \textcolor{blue}{4} & \textcolor{blue}{11} & 22 & 35 & 48 & 61 & 74 & 87 & 100 & 113 & 126 \\
2 & - & 1 & \textcolor{blue}{11} & \textcolor{blue}{37} & 99 & 249 & 573 & 1130 & 1904 & 2855 & 3975 & 5264 \\
3 & - & - & 1 & \textcolor{blue}{37} & \textcolor{blue}{225} & 845 & 2361 & 5295 & 10548 & 20611 & 39685 & 71979 \\
4 & - & - & - & 22 & \textcolor{blue}{225} & \textcolor{blue}{1377} & 5211 & 15649 & 43536 & 114529 & 266099 & 536350 \\
5 & - & - & - & 1 & 99 & \textcolor{blue}{1377} & 7658 & 32338 & 117477 & 370193 & 990566 & 2353010 \\
6 & - & - & - & - & 35 & 845 & \textcolor{blue}{8590} & 55130 & 260742 & 968743 & 2968490 & 8258521 \\
7 & - & - & - & - & 1 & 249 & 7658 & \textcolor{blue}{71823} & 442552 & 1964091 & 7107245 & 23681728 \\
8 & - & - & - & - & - & 48 & 5211 & \textcolor{blue}{71823} & 601209 & 3270704 & 14396018 & 57073451 \\
9 & - & - & - & - & - & 1 & 2361 & 55130 & \textcolor{blue}{672688} & 4631351 & 25132683 & 118210113 \\
10 & - & - & - & - & - & - & 573 & 32338 & 601209 & 5624480 & 38809778 & 216769276 \\
11 & - & - & - & - & - & - & 61 & 15649 & 442552 & \textcolor{blue}{5981484} & 52682302 & 351093447 \\
12 & - & - & - & - & - & - & 1 & 5295 & 260742 & 5624480 & 63316023 & 506479874 \\
13 & - & - & - & - & - & - & - & 1130 & 117477 & 4631351 & \textcolor{blue}{67575644} & 657393039 \\
14 & - & - & - & - & - & - & - & 74 & 43536 & 3270704 & 63316023 & 767681103 \\
15 & - & - & - & - & - & - & - & 1 & 10548 & 1964091 & 52682302 & \textcolor{blue}{807806236} \\
16 & - & - & - & - & - & - & - & - & 1904 & 968743 & 38809778 & 767681103 \\
17 & - & - & - & - & - & - & - & - & 87 & 370193 & 25132683 & 657393039 \\
18 & - & - & - & - & - & - & - & - & 1 & 114529 & 14396018 & 506479874 \\
19 & - & - & - & - & - & - & - & - & - & 20611 & 7107245 & 351093447 \\
20 & - & - & - & - & - & - & - & - & - & 2855 & 2968490 & 216769276 \\
21 & - & - & - & - & - & - & - & - & - & 100 & 990566 & 118210113 \\
22 & - & - & - & - & - & - & - & - & - & 1 & 266099 & 57073451 \\
23 & - & - & - & - & - & - & - & - & - & - & 39685 & 23681728 \\
24 & - & - & - & - & - & - & - & - & - & - & 3975 & 8258521 \\
25 & - & - & - & - & - & - & - & - & - & - & 113 & 2353010 \\
26 & - & - & - & - & - & - & - & - & - & - & 1 & 536350 \\
27 & - & - & - & - & - & - & - & - & - & - & - & 71979 \\
28 & - & - & - & - & - & - & - & - & - & - & - & 5264 \\
29 & - & - & - & - & - & - & - & - & - & - & - & 126 \\
30 & - & - & - & - & - & - & - & - & - & - & - & 1 \\
\end{array}
\]
\caption{\Cref{conj2} for $k = 3$.}
\label{table2}
\end{figure}

\section{Generalisation of the odd-length}

So far, the odd-length is only defined in type $A_n$. In this section we extend its definition to any (finite) Weyl group $(W,S)$. For convenience, we will use left-inversion sets rather than right-inversion sets. 
We also set once for all $N(w) := N_L(w)$ for any $w \in W$ and 
\begin{align}\label{def root inversion set}
    \Phi(w) := \{\alpha \in \Phi^+~|~s_{\alpha} \in N(w)\} = \{\alpha \in \Phi^+~|~w^{-1}(\alpha) \in \Phi^{-}\} = \Phi^+ \cap w(\Phi^{-}).
\end{align}

We set $|\alpha| := \alpha$ if $\alpha \in \Phi^+$ and $|\alpha| := -\alpha$ if $\alpha \in \Phi^-$. The following map is a bijection
$$
\begin{array}{ccc}
    N(w) & \longrightarrow & \Phi(w)\\
      t = wsw^{-1} & \longmapsto & \alpha_t :=|w(\alpha_s)|.
\end{array}
$$

It is well known that for any $x,y \in W$
\begin{align}\label{formula product inversion set N}
    N(xy) = N(x) ~\Delta~ xN(y)x^{-1},
\end{align}
where $\Delta$ denotes here the symmetric difference.  \Cref{formula product inversion set N} has a root analogue, given by
\begin{align}\label{formula product inversion set}
    \Phi(xy) = \Phi(x) ~\Delta~ |x(\Phi(y))|.
\end{align}

\medskip

We define the map $\mathcal{L}_{k,h}$ for any $k \in \mathbb{N}^*$ and any $h \in \mathbb{Z}$ as follows

\begin{equation}\label{odd-length general}
\begin{array}{ccccc}
      \mathcal{L}_{k,h} & : & W & \longrightarrow & \mathbb{N}  \\
                    &   & w & \longmapsto     & |\{\alpha \in \Phi(w)~|~\mathrm{ht}(\alpha) = h~\text{mod}~k\}|.
\end{array}
\end{equation}

It is also clear that 
\begin{align}\label{other des L modular}
    \mathcal{L}_{k,h}(w) = |\{t\in N(w)~|~\mathrm{ht}(\alpha_t) = h~\text{mod}~k\}|.
\end{align}

\medskip

We introduce now, on any Weyl group $W$, the analogue of $B_{n,k,1}$ but for the statistic $\mathcal{L}_{k,h}$
\begin{equation}\label{poly brenti gene}
    \mathcal{B}_{k,h}(q) := \sum\limits_{w \in W}q^{\mathcal{L}_{k,h}(w)}.
\end{equation}
\medskip

The following lemma is well known, but we give its proof for completeness.

\begin{lemma}\label{lemma inv set}
    Let $w \in W$. We have
    \begin{itemize}
        \item[(1)] $N(ww_0) = T \setminus N(w)$.
        \item[(2)] $N(w_0w) = T \setminus w_0N(w)w_0$.
    \end{itemize}
\end{lemma}

\begin{proof}
By \Cref{formula product inversion set N} we have $N(ww_0) = N(w) ~\Delta~ wN(w_0)w^{-1}$. However, $N(w_0) = T$ and then $wN(w_0)w^{-1} = T$. Therefore, $N(ww_0) = N(w) ~\Delta~ T= \left(N(w) \cup T\right) \setminus \left(N(w) \cap T\right) = T \setminus N(w)$, where the last equality comes from the fact that $N(w) \subset T$. The second point is straightforward: $N(w_0w) = N(w_0) ~\Delta~ w_0N(w)w_0^{-1} = T ~\Delta ~w_0N(w)w_0 = T \setminus w_0N(w)w_0$, since $w_0N(w)w_0 \subset T$.
\end{proof}

\medskip
We have the following properties on $\mathcal{L}_{k,h}$, similar to those in  \Cref{prop propriété type A}.
\medskip

\begin{proposition}
The map $\mathcal{L}_{k,h}$ satisfies the following properties
    \begin{itemize}
        \item[(1)] $\mathcal{L}_{k,h}(e) = 0$,
        \item[(2)] $\mathcal{L}_{k,1}(s) = 1$ for any $s \in S$,
        \item[(3)] $\mathcal{L}_{k,h}(w_0w) = \mathcal{L}_{k,h}(ww_0) = \mathcal{L}_{k,h}(w_0)-\mathcal{L}_{k,h}(w)$ for any $w \in W$,
        \item[(4)] $w_0$ is the unique element on which $\mathcal{L}_{k,h}$ reaches its maximum $\mathcal{L}_{k,h}(w_0)$.
    \end{itemize}
\end{proposition}

\begin{proof}
    Points (1) and (2) are obvious. Let $w \in W$. We first show that $\mathcal{L}_{k,h}(w_0w) = \mathcal{L}_{k,h}(w_0)-\mathcal{L}_{k,h}(w)$. By \Cref{formula product inversion set} we have
\begin{align*}
    \Phi(w_0w) & = \Phi(w_0) ~\Delta ~ |w_0(\Phi(w))| \\
               & = \Phi(w_0) ~\Delta ~ -w_0(\Phi(w)) \quad \text{because} \quad w_0(\Phi^+) = \Phi^- \\
               & = \Phi(w_0) \setminus -w_0(\Phi(w)) \quad \text{because} \quad -w_0(\Phi(w)) \subset \Phi^+ = \Phi(w_0).
\end{align*}

Therefore, it follows that
\begin{align*}
    \mathcal{L}_{k,h}(w_0w) & = |\{ \alpha \in \Phi(w_0w)~|~ \mathrm{ht}(\alpha) = h~\text{mod}~k\}| \\
    & = |\{ \alpha \in  \Phi(w_0) \setminus -w_0(\Phi(w))~|~ \mathrm{ht}(\alpha) = h~\text{mod}~k\}| \\
     & = |\{ \alpha \in  \Phi(w_0) ~|~ \mathrm{ht}(\alpha) = h~\text{mod}~k\}| - |\{ \alpha \in -w_0(\Phi(w)) ~|~ \mathrm{ht}(\alpha) = h~\text{mod}~k\}| \\
    & = \mathcal{L}_{k,h}(w_0) - |\{ \alpha \in -w_0(\Phi(w)) ~|~ \mathrm{ht}(\alpha) = h~\text{mod}~k\}|.
\end{align*}

Moreover, we also know that $\mathrm{ht}(w_0(\beta)) = -\mathrm{ht}(\beta)$ for any $\beta \in \Phi$. Therefore, it follows that 
\begin{align*}
    \{ \alpha \in -w_0(\Phi(w)) ~|~ \mathrm{ht}(\alpha) = h~\text{mod}~k\}  & = \{ -w_0(\beta) ~|~ \beta \in \Phi(w),~ \mathrm{ht}(-w_0(\beta)) = h~\text{mod}~k\} \\
    & = \{ -w_0(\beta) ~|~ \beta \in \Phi(w),~ \mathrm{ht}(\beta) = h~\text{mod}~k\}.
\end{align*}

However, it is clear that the following map is a bijection
$$
\begin{array}{ccc}
    \{ -w_0(\beta) ~|~ \beta \in \Phi(w),~ \mathrm{ht}(\beta) = h~\text{mod}~k\} & \longrightarrow & \{ \beta \in \Phi(w)~|~ \mathrm{ht}(\beta) = h~\text{mod}~k\} \\
     \gamma & \longmapsto & -w_0(\gamma).
\end{array}
$$
Since $\mathcal{L}_{k,h}(w) = |\{ \beta \in \Phi(w)~|~ \mathrm{ht}(\beta) = h~\text{mod}~k\}|$, we have shown our first equality. 

We show now the other equality: $\mathcal{L}_{k,h}(ww_0) = \mathcal{L}_{k,h}(w_0)-\mathcal{L}_{k,h}(w)$. 
By \Cref{other des L modular} and \Cref{lemma inv set} we have 
\begin{align*}
    \mathcal{L}_{k,h}(ww_0) & = |\{ t \in N(ww_0)~|~ \mathrm{ht}(\alpha_t) = h~\text{mod}~k\}| \\
    & = |\{ t \in  T \setminus N(w) ~|~ \mathrm{ht}(\alpha_t) = h~\text{mod}~k\}| \\
     & = |\{ t \in  T  ~|~ \mathrm{ht}(\alpha_t) = h~\text{mod}~k\} \setminus \{ t \in  N(w)  ~|~ \mathrm{ht}(\alpha_t) = h~\text{mod}~k\}| \\
     & = |\{t \in  T  ~|~ \mathrm{ht}(\alpha_t) = h~\text{mod}~k\}| - |\{t \in  N(w)  ~|~ \mathrm{ht}(\alpha_t) = h~\text{mod}~k\}| \\
    & = \mathcal{L}_{k,h}(w_0)-\mathcal{L}_{k,h}(w).
\end{align*}

    Point (4) is a direct consequence of (3).
\end{proof}

\bigskip

As for type $A_n$, the following question seems natural.
  \medskip
\begin{question}
    We may wonder if there is an analogue for $\mathcal{L}_{k,h}$ in terms of a map resembling to (\ref{odd-length type A}) ?
\end{question}

\medskip

\begin{conjecture}\label{conj gen}
    For any $W \neq W(F_4)$, for any $k \geq 3$, for any $h \geq 1$, the polynomial $\mathcal{B}_{k,h}(q)$ is unimodal. For $W = W(F_4)$, the polynomial $\mathcal{B}_{k,h}(q)$ is unimodal when $k \geq 4$ and for any $h \geq 1$.
\end{conjecture}

\medskip

 \begin{figure}[h!]
\[
\begin{array}{r|r|r|r}
\text{Type} & A_n & B_n & D_n  \\
\hline
n ~\text{up to} & 12 & 10 & 10  \\
\end{array}
\]
\caption{Cases where \Cref{conj gen} has been checked for all adequate choices of $k,h$.}
\end{figure}

\medskip
 
\begin{remark}
It turns out that in types $G_2,~F_4,~E_6$, the polynomial $\mathcal{B}_{2,1}(q)$ is not unimodal, for $E_7$ and $B_n$ with $n \in \{5,\dots, 13\}$ it is unimodal, and for $D_n$, it is  unimodal for $n=6$ and $n=10$, where we have checked for $n \in \{4,\dots,10\}$. We give in \Cref{table Bn} and \Cref{table Dn} the coefficients of $\mathcal{B}_{2,1}(q)$ in types $B_n$ and $D_n$.
\end{remark}

Link to GitHub repository:
\url{https://github.com/jfromentin/coxmodlen/}

\bigskip

\bibliographystyle{plain}
\bibliography{Biblio.bib}

\bigskip

\noindent Nathan Chapelier-Laget \\
Université du Littoral Côte d'Opale,\\
Calais, 62228, France\\
\texttt{nathan.chapelier@gmail.com}

\bigskip

\noindent Jean Fromentin \\
Université du Littoral Côte d'Opale,\\
Calais, 62228, France\\
\texttt{jean.fromentin@univ-littoral.fr}

\newpage

\begin{table}[h!]
\[
\small
\begin{array}{r|r|r|r|r|r|r|r|r|r|r}
d & N_{2} & N_{3} & N_{4} & N_{5} & N_{6} & N_{7} & N_{8} & N_{9} & N_{10} \\
\hline
0  & 1 & 1  & 1  & 1   & 1    & 1     & 1       & 1         & 1         \\
1  & \textcolor{blue}{3} & 7  & 11 & 15  & 19   & 23    & 27      & 31        & 35        \\
2  & \textcolor{blue}{3}  & 11 & 25 & 55  & 105  & 171   & 253     & 351       & 465       \\
3  & 1 & \textcolor{red}{10} & 52 & 166 & 344  & 646   & 1168    & 1986      & 3156      \\
4  & - & 11 & 70 & 274 & 762  & 1998  & 4466    & 8494      & 15018     \\
5  & - & 7  & \textcolor{red}{66} & 367 & 1387 & 4365  & 11043   & 25025     & 53735     \\
6  & - & 1  & 70 & 507 & 2442 & 9168  & 50151   & 66836     & 160560    \\
7  & - & -  & 52 & \textcolor{blue}{535}  & 3384 & 15287 & 88018   & 153295    & 410651    \\
8  & - & -  & 25 & \textcolor{blue}{535}  & 4118 & 21931 & 142191  & 308066    & 905126    \\
9  & - & -  & 11 & 507 & 5073 & 31544 & 220588  & 559876    & 1823644   \\
10 & - & -  & 1  & 367 & \textcolor{blue}{5405} & 40607 & 309201  & 989001    & 3440736   \\
11 & - & -  & -  & 274 & \textcolor{blue}{5405} & 48783 & 401116  & 1552734   & 6052617   \\
12 & - & -  & -  & 166 & 5073 & 55590 & 517149  & 2261174   & 9973305   \\
13 & - & -  & -  & 55  & 4118 & 60632 & 621804  & 3227833   & 15602907  \\
14 & - & -  & -  & 15  & 3384 & \textcolor{blue}{63628} & 712313  & 4338023   & 23281437  \\
15 & - & -  & -  & 1   & 2442 & 60632 & 786655  & 5611753   & 33719508  \\
16 & - & -  & -  & -   & 1387 & 55590 & 837237  & 6953616   & 46357542  \\
17 & - & -  & -  & -   & 762  & 48783 & \textcolor{blue}{864398}  & 8358201   & 61168114  \\
18 & - & -  & -  & -   & 344  & 40607 & 837237  & 9758875   & 79108275  \\
19 & - & -  & -  & -   & 105  & 31544 & 786655  & 10968189  & 98758261  \\
20 & - & -  & -  & -   & 19   & 21931 & 712313  & 11994272  & 119860063 \\
21 & - & -  & -  & -   & 1    & 15287 & 621804  & 12655795  & 141245583 \\
22 & - & -  & -  & -   & -    & 9168  & 517149  & \textcolor{blue}{13103853}  & 162653469 \\
23 & - & -  & -  & -   & -    & 4365  & 401116  & \textcolor{blue}{13103853}  & 183178248 \\
24 & - & -  & -  & -   & -    & 1998  & 309201  & 12655795  & 200648380 \\
25 & - & -  & -  & -   & -    & 646   & 220588  & 11994272  & 214964837 \\
26 & - & -  & -  & -   & -    & 171   & 142191  & 10968189  & 224414283 \\
27 & - & -  & -  & -   & -    & 23    & 88018   & 9758875   & \textcolor{blue}{230145644} \\
28 & - & -  & -  & -   & -    & 1     & 50151   & 8358201   & \textcolor{blue}{230145644} \\
29 & - & -  & -  & -   & -    & -     & 11043   & 6953616   & 224414283 \\
30 & - & -  & -  & -   & -    & -     & 4466    & 5611753   & 214964837 \\
31 & - & -  & -  & -   & -    & -     & 1168    & 4338023   & 200648380 \\
32 & - & -  & -  & -   & -    & -     & 253     & 3227833   & 183178248 \\
33 & - & -  & -  & -   & -    & -     & 27      & 2261174   & 162653469 \\
34 & - & -  & -  & -   & -    & -     & 1       & 1552734   & 141245583 \\
35 & - & -  & -  & -   & -    & -     & -       & 989001    & 119860063 \\
36 & - & -  & -  & -   & -    & -     & -       & 559876    & 98758261  \\
37 & - & -  & -  & -   & -    & -     & -       & 308066    & 79108275  \\
38 & - & -  & -  & -   & -    & -     & -       & 153295    & 61168114  \\
39 & - & -  & -  & -   & -    & -     & -       & 66836     & 46357542  \\
40 & - & -  & -  & -   & -    & -     & -       & 25025     & 33719508  \\
41 & - & -  & -  & -   & -    & -     & -       & 8494      & 23281437  \\
42 & - & -  & -  & -   & -    & -     & -       & 1986      & 15602907  \\
43 & - & -  & -  & -   & -    & -     & -       & 351       & 9973305   \\
44 & - & -  & -  & -   & -    & -     & -       & 31        & 6052617   \\
45 & - & -  & -  & -   & -    & -     & -       & 1         & 3440736   \\
46 & - & -  & -  & -   & -    & -     & -       & -         & 1823644   \\
47 & - & -  & -  & -   & -    & -     & -       & -         & 905126    \\
48 & - & -  & -  & -   & -    & -     & -       & -         & 410651    \\
49 & - & -  & -  & -   & -    & -     & -       & -         & 160560    \\
50 & - & -  & -  & -   & -    & -     & -       & -         & 53735     \\
51 & - & -  & -  & -   & -    & -     & -       & -         & 15018     \\
52 & - & -  & -  & -   & -    & -     & -       & -         & 3156      
\end{array}
\]
\caption{Coefficients of $\mathcal{B}_{2,1}(q)$ in type $B_n$.}
\label{table Bn}
\end{table}

\newpage

\begin{table}[h!]
\[
\small
\begin{array}{r|r|r|r|r|r|r|r|r|r}
d & N_{4} & N_{5} & N_{6} & N_{7} & N_{8} & N_{9} & N_{10}  \\
\hline
0  & 1   & 1    & 1       & 1       & 1       & 1       & 1       \\
1  & 14  & 18   & 22      & 26      & 30      & 34      & 38      \\
2  & 16  & 48   & 112     & 194     & 288     & 398     & 524     \\
3  & 50  & 194  & 422     & 696     & 1242    & 2170    & 3506    \\
4  & \textcolor{red}{30}  & 199  & 671     & 2037    & 5163    & 9911    & 16891   \\
5  & 50  & 364  & 1272    & 4108    & 10940   & 24948   & 56632   \\
6  & 16  & \textcolor{red}{272}  & 2009    & 10196   & 30960   & 73892   & 177492  \\
7  & 14  & 364  & 2474    & 12638   & 45108   & 148792  & 444948  \\
8  & 1   & 199  & 2967    & 21879   & 90290   & 319795  & 969675  \\
9  & -   & 194  & \textcolor{blue}{3140}    & 24296   & 118020  & 534758  & 1921582 \\
10 & -   & 48   & 2967    & 33354   & 205208  & 1034454 & 3722137 \\
11 & -   & 18   & 2474    & \textcolor{red}{31964}   & 229324  & 1384762 & 6004440 \\
12 & -   & 1    & 2009    & 39782   & 346165  & 2224944 & 9881508 \\
13 & -   & -    & 1272    & \textcolor{red}{31964}   & 348750  & 2752614 & 14576416 \\
14 & -   & -    & 671     & 33354   & 463064  & 3901050 & 21446815 \\
15 & -   & -    & 422     & 24296   & \textcolor{red}{426234}  & 4403418 & 29859470 \\
16 & -   & -    & 112     & 21879   & 519386  & 5755692 & 39762140 \\
17 & -   & -    & 22      & 12638   & \textcolor{red}{426234}  & 5942246 & 51028466 \\
18 & -   & -    & 1       & 10196   & 463064  & 7148158 & 63595343 \\
19 & -   & -    & -       & 4108    & 348750  & \textcolor{red}{6924658} & 76383568 \\
20 & -   & -    & -       & 2037    & 346165  & 7723890 & 89380382 \\
21 & -   & -    & -       & 696     & 229324  & \textcolor{red}{6924658} & 101536232 \\
22 & -   & -    & -       & 194     & 205208  & 7148158 & 111199760 \\
23 & -   & -    & -       & 26      & 118020  & 5942246 & 119505284 \\
24 & -   & -    & -       & 1       & 90290   & 5755692 & 124333732 \\
25 & -   & -    & -       & -       & 45108   & 4403418 & \textcolor{blue}{126331636} \\
26 & -   & -    & -       & -       & 30960   & 3901050 & 124333732 \\
27 & -   & -    & -       & -       & 10940   & 2752614 & 119505284 \\
28 & -   & -    & -       & -       & 5163    & 2224944 & 111199760 \\
29 & -   & -    & -       & -       & 1242    & 1384762 & 101536232 \\
30 & -   & -    & -       & -       & 288     & 1034454 & 89380382 \\
31 & -   & -    & -       & -       & 30      & 534758  & 76383568 \\
32 & -   & -    & -       & -       & 1       & 319795  & 63595343 \\
33 & -   & -    & -       & -       & -       & 148792  & 51028466 \\
34 & -   & -    & -       & -       & -       & 73892   & 39762140 \\
35 & -   & -    & -       & -       & -       & 24948   & 29859470 \\
36 & -   & -    & -       & -       & -       & 9911    & 21446815 \\
37 & -   & -    & -       & -       & -       & 2170    & 14576416 \\
38 & -   & -    & -       & -       & -       & 398     & 9881508  \\
39 & -   & -    & -       & -       & -       & 34      & 6004440  \\
40 & -   & -    & -       & -       & -       & 1       & 3722137  \\
41 & -   & -    & -       & -       & -       & -       & 1921582  \\
42 & -   & -    & -       & -       & -       & -       & 969675   \\
43 & -   & -    & -       & -       & -       & -       & 444948   \\
44 & -   & -    & -       & -       & -       & -       & 177492   \\
45 & -   & -    & -       & -       & -       & -       & 56632    \\
46 & -   & -    & -       & -       & -       & -       & 16891    \\
47 & -   & -    & -       & -       & -       & -       & 3506     \\
48 & -   & -    & -       & -       & -       & -       & 524      \\
49 & -   & -    & -       & -       & -       & -       & 38       \\
50 & -   & -    & -       & -       & -       & -       & 1        \\
\end{array}
\]
\caption{Coefficients of $\mathcal{B}_{2,1}(q)$ in type $D_n$.}
\label{table Dn}
\end{table}

\end{document}